\documentclass[reqno,12pt]{amsart}

\usepackage[T1]{fontenc}
\usepackage{amssymb}
\usepackage{amssymb, latexsym}
\usepackage{hyperref}
\usepackage{amsmath}
\usepackage{amscd}
\usepackage{enumerate}
\usepackage{amsfonts}
\usepackage{graphicx}
\usepackage[all,cmtip]{xy}
\usepackage{mathrsfs}
\usepackage{slashed}
\usepackage{xcolor}
\usepackage{bm}
\usepackage[numbers,sort&compress]{natbib}
\usepackage{appendix}

\def\ge{\geqslant}

\def\geq{\geqslant}
\def\leq{\leqslant}

\allowdisplaybreaks

\renewcommand{\Im}{\mathop{\mathrm{Im}}}

\newcommand{\C}{{\mathbb{C}}}
\newcommand{\R}{{\mathbb{R}}}
\newcommand{\N}{{\mathbb{N}}}

\newcommand{\eps}{{\varepsilon}}

\newcommand{\HHH}{\mathcal{H}}
\newcommand{\KKK}{\mathcal{K}}

\newcommand{\MMM}{\mathcal{M}}
\newcommand{\EEE}{\mathcal{E}}
\newcommand{\QQQ}{\mathcal{Q}}

\makeatletter

\newcommand{\Rmnum}[1]{\expandafter\@slowromancap\romannumeral #1@}
\makeatother

\theoremstyle{plain}
\newtheorem{theorem}{Theorem}
\newtheorem{proposition}[theorem]{Proposition}
\newtheorem{lemma}[theorem]{Lemma}

\theoremstyle{definition}

\newtheorem{remark}[theorem]{Remark}

\numberwithin{equation}{section}
\numberwithin{theorem}{section}

\numberwithin{equation}{section}

\allowdisplaybreaks
\begin{document}


\title[The dynamics of the focusing NLH with a potential]{The dynamics of the focusing NLH with a potential beyond the mass-energy threshold }

\author[S. Ji]{Shuang Ji}
\address{College of Science, China Agricultural University, \ Beijing, \ China, \ 100193}
\email{jishuang@cau.edu.cn}

\author[J. Lu]{Jing Lu}
\address{College of Science, China Agricultural University, \ Beijing, \ China, \ 100193 }
\email{lujing326@126.com}

\author[F. Meng]{Fanfei Meng}
\address{Qiyuan Lab, \ Tsinghua University, \ Beijing, \ China, \ 100095}
\email{mengfanfei17@gscaep.ac.cn}


\date{\today}


\begin{abstract}\noindent In this paper, 
we study the dynamics of the focusing nonlinear Hartree equation with a Kato potential 
$$
i\partial_t u +\Delta u - Vu = -(|\cdot|^{-\gamma} \ast |u|^2)u, \quad x \in \R^d
$$
under some assumptions on the potential $V$. We prove the blow up versus global existence dichotomy for
solutions beyond the threshold, based on the method from Duyckaerts-Roudenko \cite{dr}. 
 Furthermore, our result compensates for the one of in \cite{jl} below that threshold.

\end{abstract}

\maketitle

\section{Introduction}
In this paper, we study the Cauchy problem of the focusing nonlinear Hartree equation with a potential 
\begin{align}\label{NLHv}\tag{$\text{NLH}_{\text{V}}$}
    \left\{ \aligned
   & i\partial_t u +\Delta u - Vu = -(|\cdot|^{-\gamma} \ast |u|^2)u, \\
   & u(0)=u_0\in {H}_x^1(\R^d),
   \endaligned
  \right. \qquad (t, x) \in \R \times \R^d,
\end{align}
where $ u: \R \times \R^d \rightarrow \C $ is the wave function, $ V : \R^d \to \R $ is a real-valued potential and $ \ast $ denotes the convolution of spacial variable. 
Here we consider the energy-subcritical case, that is $2<\gamma< \min\{4,d\}$. 


We denote that $\HHH := -\Delta + V$ and
\begin{equation*}
    \Vert f \Vert_{\dot{H}_V^1(\R^d)}^2:=<\HHH f,f>=\int_{\R^d} | \nabla f|^2dx+\int_{\R^d} V|f|^2dx,\quad f\in H_x^1(\R^d).
\end{equation*}

The solution to \eqref{NLHv} satisfies the laws of mass conservation and energy conservation, which can be expressed respectively by 
\begin{equation}\label{mass}
	M(u) = \int_{\R^d} \vert u(t,x) \vert ^2 dx = M(u_0),
\end{equation}
\begin{equation}\label{energy}
	E_V(u) = \frac{1}{2} \Vert u \Vert_{\dot{H}_V^1(\R^d)}^2  - \frac{1}{4}P(u) = E_V(u_0),
\end{equation}
where
\begin{equation}\label{momentum}
    P(u)=\int_{\R^d}\int_{\R^d} \frac{|u(x)|^2 |u(y)|^2}{|x-y|^{\gamma}}dxdy=\int_{\R^d} (|x|^{-\gamma} \ast |u|^2)|u|^2 dx.
\end{equation} 




Before stating our main results, we first recall the literature of the focusing nonlinear Hartree equations with no potential
\begin{align}\label{NLH0}\tag{$\text{NLH}_{\text{0}}$}
    i\partial_t u +\Delta u = \mu(|\cdot|^{-\gamma} \ast |u|^2)u, \quad 2<\gamma <\min \{ 4,d \},\quad x \in \R^d,
\end{align}
which is called focusing if $ \mu < 0 $ and defocusing if $ \mu > 0 $. 

In \cite{gw}, Gao-Wu investigated the focusing $\eqref{NLH0}$ for $\gamma=3$ with $u_0 \in H_x^1(\R^d)$  by concentration compactness method. 
Denote the mass-energy threshold as 
$M (Q)E_0(Q),$ 
where $Q$ is the ground state of the equation \eqref{NLH0} and
$$E_0(u)
=\frac{1}{2} \int_{\R^d}|\nabla u|^2dx - \frac{1}{4}P(u) 
=E_0(u_0).$$
Their results are as follows.
\begin{theorem}[Sub-threshold dynamics for $\eqref{NLH0}$, Gao-Wu \cite{gw}]
    For $(\gamma,d)=(3,5)$, let the radial $u_0 \in H_x^1(\R^d)$ and satisfy
    $M (u_0)E_0(u_0) < M (Q)E_0(Q)$.
    \begin{enumerate}[\rm (i)]
      \item  If $\Vert \nabla u_0 \Vert_{L_x^2(\R^d)}||u_0||_{L_x^2(\R^d)}<||\nabla Q||_{L_x^2(\R^d)}||Q||_{L_x^2(\R^d)}$,
        then the solution is global and scatters. 
	\item  If $\Vert \nabla u_0 \Vert_{L_x^2(\R^d)}||u_0||_{L_x^2(\R^d)} > ||\nabla Q||_{L_x^2(\R^d)}||Q||_{L_x^2(\R^d)}$,
	then the solution will blow up in finite time.
 \end{enumerate}
\end{theorem}

For the case beyond the threshold, Yang-Li \cite{ylwc} have established a dichotomy of blow up versus global existence for $\eqref{NLH0}$ by the method from Duyckaerts-Roudenko \cite{dr}. 
Their results are as follows.
\begin{theorem}[Super-threshold dynamics for $\eqref{NLH0}$, Yang-Li \cite{ylwc}] 
For $s_c=\frac{\gamma-2}{2}$ and $d \geq 3$, assume $I(0)<\infty$, $M (u_0)E_0(u_0) > M (Q)E_0(Q)$ and 
    \begin{align*}
    \biggl(\frac{M (u_0)E_0(u_0)}{M (Q)E_0(Q)}\biggr) \biggl( 1 - \frac{(I'(0))^2}{32E_0(u_0)I(0)} \biggr)\leq 1.
\end{align*}
\begin{enumerate}[\rm (i)]
    \item If $ \bigl[M(u_0)\bigr]^{1-s_c}\bigl[P(u_0)\bigr]^{s_c} > \bigl[M(Q)\bigr]^{1-s_c}\bigr[P(Q)\bigl]^{s_c}$ and $I'(0)\leq 0$,
    then the solution blows up in finite time.
    \item If $\bigl[M(u_0)\bigr]^{1-s_c}\bigl[P(u_0)\bigr]^{s_c} < \bigl[M(Q)\bigr]^{1-s_c}\bigr[P(Q)\bigl]^{s_c}$ and $I'(0) \geq 0$,
    then the solution exists globally. 
\end{enumerate}
\end{theorem}


For $\eqref{NLH0}$, Miao-Xu-Zhao studied the well-posedness, ill-posedness, the sharp local well-posedness and the global existence in \cite{mxz2}.
After that, they further established the global well-posedness and scattering criteria when $\gamma=2$ below the threshold \cite{mxz1}.
Li-Miao-Zhang \cite{lmz} proved that the solution scatters in both time directions when $\gamma=4$ and $d \geq 5$.
A similar result was obtained for spherically symmetric initial data in \cite{mxz}.
Meng \cite{meng} used the radial Sobolev embedding and a virial-Morawetz type estimate to study the scattering result of $\eqref{NLH0}$ with $(\gamma,d)=(3,5)$,
instead of Kenig-Merle's concentrated compactness methods \cite{KM,KM1}. 
In particular, Miao-Wu-Xu \cite{mwx} continued the study on the dynamics of the radial solutions at the  energy threshold.


Recently, more and more mathematicians have been considerably interested in the dispersive equations with different potentials, 
which are of paramount physical importance. 
For example, the Coulomb potential describes the coulomb force between two charged particles in quantum mechanical terms, 
and the inverse-square potential reveals that the intensity of the action between particles decays linearly with the square of the distance.

There are also many studies on the focusing nonlinear Schr\"odinger equations with a potential $V$.
Firstly, Rodnianski-Schlag \cite{rs} established the dispersive estimates for solutions to the linear Schr\"odinger equation in 3D with a time-dependent potential. 
Later, more and more mathematicians studied the nonlinear Schr\"odinger equation with a real-valued potential $V$ 
\begin{align}\label{NLSv}\tag{$\text{NLS}_{\text{V}}$}
    i\partial_t u +\Delta u -Vu + |u|^{p-1}u=0, 
    \quad x \in \R^d.
\end{align}
For the case that $V$ is the inverse-square potential $a /|x|^2$ in \eqref{NLSv}, 
Killip-Murphy-Visan-Zheng \cite{kmvz} obtained the scattering and blow up dichotomy below the threshold with $a>-\frac{1}{4}$ and $p=3$ in $\R^3$. 
Then Lu-Miao-Murphy \cite{lmm} extended their results to dimensions $3 \leq d \leq 6$.
Based on the method from Duyckaerts-Roudenko \cite{dr}, Deng-Lu-Meng \cite{dlm1} further extended the results of \cite{kmvz} to beyond the threshold.
For the case that $V$ is the Coulomb potential $K|x|^{-1}$ in \eqref{NLSv}, Miao-Zhang-Zheng \cite{mzz} proved the global existence when $K>0$, and the scattering theory when $K<0$ for in $\R^3$.

In particular, we focus on the case that $V$ is the Kato potential with a small negative part denoted as 
\[V_{-}(x):= \min\{V(x),0\}. \]
To be precise, we define that the potential class $\KKK_0$ is the closed space of bounded and compactly supported functions endowed with the Kato norm
\[
\Vert V \Vert_{\KKK} := \Vert (-\Delta)^{-1}V) \Vert_{L_x^{\infty}}.
\]
In order to prove the scattering result for \eqref{NLSv} with $(p,d)=(3,3)$ below the mass-energy threshold, Hong used the concentration compactness method and introduced the following customized assumptions of Kato potential V in \cite{hong} 
\begin{equation}\label{v31}
    V \in \KKK_0 \cap L_x^{\frac{3}{2}}(\R^3),
\end{equation}
and 
\begin{equation}\label{v32}
    \Vert V_{-}(x) \Vert_{\KKK} < 4\pi,
\end{equation} 
Hamano-Ikeda \cite{hi} studied the energy scattering below the threshold from Dodson-Murphy \cite{dm} and the blow up criteria based on an argument of Du-Wu-Zhang \cite{dwz} when $p>1$.
Later, Wang \cite{wangying} investigated the blow up and scattering results of \eqref{NLSv} beyond the mass-energy threshold when $p=3$.


Inspired by their results, we want to expand the range of $V$ to higher dimensions $d \geq 3$. We 
assume that the Kato potential $V$ satisfies
\begin{equation}\label{v0}
    V \in \KKK_0 \cap L_x^{\frac{d}{2}}(\R^d),
\end{equation}
and
\begin{equation}\label{v1}
    \Vert V_{-}(x) \Vert_{\KKK} < \frac{1}{C_d},
\end{equation}
where $C_d=\Gamma(\frac d2)\big/\bigl[(2-d)2\pi^{\frac d2 }\bigr]$.
\begin{remark}
    For the sake of clarity, we explain why the Kato potential $V$ need to satisfy \eqref{v0} and \eqref{v1}.
    $V\in\KKK_0$ and $\Vert V_{-}(x) \Vert_{\KKK} < 1/C_d$ ensure that the operator $\HHH$ is positive, which implies that $\HHH$ is equivalent to $-\Delta$ in the sense of 
    \begin{align*}
        (1- C_d\Vert V_{-} \Vert_{\KKK})\Vert \nabla u \Vert_{L_x^2(\R^d)}^2 
        \leq \Vert \HHH^{\frac12} u \Vert_{L_x^2(\R^d)}^2 
        \leq  (1 + C_d\Vert V \Vert_{\KKK})\Vert \nabla u \Vert_{L_x^2(\R^d)}^2 
    \end{align*}
    (see the specific proof in Lemma \ref{A1}).
    $V\in L_x^{d/2}(\R^d)$ is an additional condition for proving such the equivalence of the Sobolev norm in the sense of
    $
    \|f\|_{{\dot{H}}_V^{1}(\R^d)} \sim \|f\|_{\dot{H}_x^{1}(\R^d)}
    $
    when $ d \ge 3 $ (see details in \cite{hong}).
    Then from Beceanu-Goldberg \cite{bg}, the operator $\HHH$ has no eigenvalues or nonnegative resonance.
    Therefore the dispersive estimates and Strichartz estimates are valid by Ionescu-Jerison \cite{ij}, providing a robust foundation for the dichotomy of blow up versus global existence.
\end{remark}

Next we review our previous work and demonstrate our main results in this paper.
For completeness and logic, let’s start with a discussion of the Kato potential $V$ and the mass-energy threshold. 
Hong \cite{hong} has shown a stimulating discussion about the existence of ground state in \eqref{NLSv}. Later, for the nonlinear Hartree equation with the inverse-potential,  
Chen-Lu-Meng \cite{clm} also discussed the existence of ground state. 
Based on their methods,  we give the similar variational characterization of \eqref{NLHv} as follows.
\begin{proposition}[Variational properties]\label{Pq}
Suppose $V$ satisfies \eqref{v0} and \eqref{v1}.
\begin{enumerate}[\rm (i)]
    \item If $V_{-}=0$, then the sequence ${Q(\cdot -n)}_{n\in \N}$ maximizes $W(u)$, where $Q$ is the ground state of the elliptic equation
    \begin{equation}\label{elliptic}
        \Delta Q-Q+(|\cdot|^{-\gamma} \ast |Q|^2)Q=0,
    \end{equation}
    and $$W(u)=\frac{P(u)}{\Vert u \Vert_{\dot{H}_V^1(\R^d)}^{\gamma} \Vert u \Vert_{L_x^2(\R^d)}^{4-\gamma}}.$$
    \item If $V_{-} \neq 0$, then there exists a maximizer $\QQQ \in H_x^1(\R^d)$ solving the elliptic equation
    \begin{align}\label{elliptic1}
      (-\Delta+V) \QQQ + w_{\QQQ}^2 \QQQ - (|\cdot|^{-\gamma} \ast |\QQQ|^2)|\QQQ=0, 
      \quad w_{\QQQ}^2= \frac{(4-\gamma)\Vert \QQQ \Vert_{\dot{H}_V^1(\R^d)}^2}{\gamma\Vert \QQQ \Vert_{L_x^2(\R^d)}^2}.
    \end{align}
\end{enumerate}
\end{proposition}

We denote that 
$$
\MMM \EEE =\left\{\begin{aligned}
    &\frac{M(u)^{1-s_c}E_V(u)^{s_c}}{M(Q)^{1-s_c}E_0(Q)^{s_c}}, \quad & if~V_-=0,\\
    &\frac{M(u)^{1-s_c}E_V(u)^{s_c}}{M(\QQQ)^{1-s_c}E_V(\QQQ)^{s_c}}, \quad & if~V_- \neq 0.
\end{aligned}\right. 
$$    
In \cite{jl}, we have proved the scattering  theory and the blow up result of (NLH) with Kato potential below the threshold. The results are as follows:
\begin{theorem}[\cite{jl}]\label{jl}
For $ ( \gamma, d) = ( 3, 5) $ in \eqref{NLHv}, we assume $ V \ge 0 $ and satisfies \eqref{v0} and \eqref{v1}. 
Let $ u_0 \in H_x^1(\R^d)$ be radial and satisfy $\MMM \EEE <1$.
\begin{enumerate}[\rm (i)]
    \item If  the initial data  $u_0$ satisfies $\Vert  u_0 \Vert_{\dot{H}_V^1(\R^d)}||u_0||_{L_x^2(\R^d)}<\Vert Q \Vert_{\dot{H}_x^1(\R^d)}||Q||_{L_x^2(\R^d)},$
    then $\Vert  u\Vert_{\dot{H}_V^1(\R^d)}||u||_{L_x^2(\R^d)}<\Vert Q \Vert_{\dot{H}_x^1(\R^d)}||Q||_{L_x^2(\R^d)}.$
    Moreover, if $\ x \cdot \nabla V \leq 0 $ and $ x \cdot \nabla V \in L_x^{\frac52}(\R^d) $, then the global solution scatters in $H_x^1(\R^d)$  in both time directions.
    \item If  the initial data  $u_0$ satisfies $\Vert  u_0 \Vert_{\dot{H}_V^1(\R^d)}||u_0||_{L_x^2(\R^d)}>\Vert Q \Vert_{\dot{H}_x^1(\R^d)}||Q||_{L_x^2(\R^d)},$
    then $\Vert  u\Vert_{\dot{H}_V^1(\R^d)}||u||_{L_x^2(\R^d)}>\Vert Q \Vert_{\dot{H}_x^1(\R^d)}||Q||_{L_x^2(\R^d)}$ 
    during the maximal existence time.
    Moreover, if $\ x \cdot  \nabla V \in L_x^{\frac52}(\R^d)$, $2V+ x\cdot \nabla V \geq 0$, then either 
    $T^*<+\infty$ and 
    $$\lim_{t \to T^*}\Vert \nabla u(t)\Vert_{L_x^2(\R^d)}=\infty,$$ 
    or $T^*=+\infty$ and there exists a time sequence $t_n \to +\infty$ such that 
    $$\lim_{n \to +\infty } \Vert \nabla u(t_n)\Vert_{L_x^2(\R^d)}=\infty.$$
\end{enumerate}  
\end{theorem}
\begin{remark}
    In \cite{jl}, we dismiss the influence of $V_-$ for the assumption that $V\geq0$. 
    So the ground state $\QQQ$ in the result here satisfies the elliptic equation \eqref{elliptic}, which implies that $\QQQ$ is equivalent to $Q$ under the condition $V \geq 0$.  
    Thus we can rewrite the condition $\MMM\EEE<1$ as $E_V(u_0)M(u_0)<E_0(Q)M(Q)$ in Theorem \ref{jl}.
\end{remark}

Inspired by this result, we want to know what will happen when beyond the threshold. 
Here we define that
\begin{align*}
    I(t)=\int_{\R^d}|x|^2|u(t,x)|^2dx,
    \quad xu \in L_x^2(\R^d).
\end{align*}
The main theorem of this paper can be concluded as follows.
\begin{theorem}\label{Tconclusion}
Assume that the Kato potential V satisfies \eqref{v0} and \eqref{v1} in \eqref{NLHv} with $d \geq 3$.
Suppose $u_0 \in \Sigma:=\left\{f \in H_x^1(\R^d): xf \in L_x^2(\R^d)\right\}$, $I(0)< \infty $, $\MMM \EEE >1$	 and 
\begin{align}\label{me<1}
    \MMM \EEE \biggl( 1 - \frac{(I'(0))^2}{32E_V(u_0)I(0)} \biggr)\leq 1,
\end{align}
\begin{enumerate}[\rm (i)]
    \item {\rm (Blow up)} If 
    $I'(0)\leq 0$, $2V+x\cdot \nabla V \geq 0$, and
    \begin{align}\label{mpu>mpq}
        \bigl[M(u_0)\bigr]^{1-s_c}\bigl[P(u_0)\bigr]^{s_c} > \bigl[M(\QQQ)\bigr]^{1-s_c}\bigr[P(\QQQ)\bigl]^{s_c},
    \end{align} 
    then $u(t,x)$ blows up in finite time.
    \item {\rm (Global existence)} If 
    $I'(0) \geq 0$, $2V+x\cdot \nabla V \leq 0$, and
    \begin{align}\label{mpu<mpq}
        \bigl[M(u_0)\bigr]^{1-s_c}\bigl[P(u_0)\bigr]^{s_c} < \bigl[M(\QQQ)\bigr]^{1-s_c}\bigr[P(\QQQ)\bigl]^{s_c},
    \end{align}
    then $u(t,x)$ exists globally. 
    Moreover, 
    \begin{align}\label{sup}
        \lim_{t\to \infty} \sup \bigl[ M(u)\bigr]^{1-s_c}\bigl[P(u) \bigr]^{s_c} < \bigl[M(\QQQ) \bigr]^{1-s_c}\bigl[P(\QQQ) \bigr]^{s_c}.
    \end{align}
\end{enumerate}
\end{theorem}

\textbf{The sketch of blow up:} 
Based on  the classical argument from Glassey in \cite{glassey},  if the initial data  $u_0\in \Sigma$, then the solution will blow up in finite time.  
Then we  simplify the problem to prove  that for $xu \in L_x^2(\R^d)$,
\begin{align*}
    \frac{d^2}{dt^2}I(t)= \frac{d^2}{dt^2}\Vert xu \Vert_{L_x^2(\R^d)}^2 < 0, \quad \forall ~ t \in \left[ 0,T_+(u) \right).
\end{align*}
We make use of $z(t)=\sqrt{I(t)}$, and prove the claim that $z''(t) < 0$ in finite time (see details in Section 3.1). 
So we derive that $u$ will blow up in finite time.

\textbf{The sketch of  global existence:} 
By the contradiction argument, we prove the lower bound of $z'(t)$ (see details in Section 3.2). 
Then we use the claim to prove the boundedness \eqref{sup}. 
Combing it with the mass and energy conservation, we eventually obtain the boundedness of  $\Vert u \Vert_{\dot{H}_V^1(\R^d)}$, which implies that the solution $u$ exists globally. 


\begin{remark}
    The condition $I(0)<\infty$ ensures the existence of $I'(0)$ and $I''(0)$ of virial identity, which play important role of proving our main results. 
    The condition $\MMM \EEE >1$ implies that our paper  is based on \eqref{NLHv} beyond the threshold.
    The condition \eqref{me<1} provides an unified estimate of $z'(0)$, which facilitates the proof of blow up and global existence.
    The condition $I'(0)\leq 0$ is instrumental in the proof of blow up.
    The condition $I'(0)\geq 0$ is used to prove the claim, which is essential for demonstrating global existence.
    The various forms of the expression $2V+x\cdot \nabla V$ are used to estimate the potential term in $I''(t)$ in different conditions. 
    
\end{remark}
\textbf{Outline of this paper:} 
In Section 2, we introduce the local well-posedness, the positivity of the operator $\HHH$,  variational characterization and the virial identity.  In Section 3, with several claims and analysis, we prove our main results, including blow up and global existence. 

\section{Preliminaries}
In this section, we introduce the notation and several fundamental lemmas needed in this paper.
The notation $A\lesssim B$ means that
${A}\leqslant{CB}$ for some constant $C>0$.
Likewise,  if $A\lesssim B \lesssim A$, we say that $A \sim B$. 
We use $L_x^{r}(\mathbb{R}^d)$ to denote the Lebesgue space of functions $f:\mathbb{R}^{d}\rightarrow{\mathbb{C}}$ whose norm
$$\|f\|_{L_x^{r}}:=\Big(\int_{\mathbb{R}^d}|f(x)|^{r}dx\Big)^{\frac{1}{r}}$$
is finite, with the usual modifications when $r=\infty$.

\subsection{Local well-posedness}
In \cite{jl}, we have introduced the Strichartz estimates of 
\eqref{NLHv} and proved the local well-posedness by Banach contraction mapping principle. 
Here we only demonstrate the local well-posedness as follows.
\begin{lemma}[Local well-posedness, \cite{jl}]
    Let \ $V: \R^d \rightarrow \R$ satisfy  \eqref{v0} and \eqref{v1}. Then the equation \eqref{NLHv} is locally well-posed in $H_x^1(\R^d)$ for $(\gamma, d)=(3,5)$.
\end{lemma}

With the depiction of local well-posedness, we can further discuss the global existence and blow up of \eqref{NLHv}. 
Indeed, if we want to discuss the local well-posedness when $2<\gamma<4$ rather than $\gamma=3$ in \cite{jl}, we only need to make some minor alterations to the proof. Then we can also obtain the similar result. 
Here we omit the proof.

\subsection{The positivity of $\HHH$}\label{A1}
In this subsection, we discuss the Kato potential $V$ and the operator $\HHH$. 
Indeed, the operator $\HHH$ is positive definite when the negative part of $V$ is sufficiently small.
We summarize this as the lemma below.
\begin{lemma}
    For $C_d=\Gamma(\frac d2)\big/\bigl[(2-d)2\pi^{\frac d2 }\bigr]$, 
    if $V\in \KKK$, then
    \begin{align*}
        \int_{\R^d} V |u|^2 dx \leq C_d \Vert V \Vert_{\KKK} \Vert \nabla u \Vert_{L_x^2(\R^d)}^2.
    \end{align*}
    In particular, if $ \Vert V_{-} \Vert_{\KKK} < \frac{1}{C_d}$, we have 
    \begin{align*}
        (1- C_d\Vert V_{-} \Vert_{\KKK})\Vert \nabla u \Vert_{L_x^2(\R^d)}^2 
        \leq \Vert \HHH^{\frac12} u \Vert_{L_x^2(\R^d)}^2 
        = \int_{\R^d} \HHH u \bar{u} dx
        \leq  (1 + C_d\Vert V \Vert_{\KKK})\Vert \nabla u \Vert_{L_x^2(\R^d)}^2. 
    \end{align*}
\end{lemma}
\begin{proof}
    We first consider the fundamental solution $K(x)$ of the Laplace equation $-\Delta u = 0$, that is $\Delta K(x)=\delta(x)$.
    Then for the equation  $-\Delta u =f$, we have 
    \begin{align*}
        -\Delta u=-\Delta (-\Delta)^{-1}f= -\Delta (-K \ast f ) = -\Delta (-K) \ast f=f.
    \end{align*}
    For $d\geq 3$, the fundamental solution of the Laplace equation is 
    $K(x)=C_d|x|^{2-d}$ ,
    where
    \begin{align*}
        C_d=\frac{\Gamma (\frac d2)}{(d-2)2\pi^{\frac d2}}.
    \end{align*}
    Thus we find the relation $(-\Delta)^{-1}f= -K \ast f=\int_{\R^d}C_d\frac{f(y)}{|x-y|^{d-2}}dy$.
    Know that
    \begin{align*}
        \bigl\Vert |V|^{\frac 12}(-\Delta)^{-1}|V|^{\frac 12}u\bigr\Vert_{L_x^2(\R^d)}^2
        &=\int_{\R^d}|V(x)| \biggl| C_d\frac{|V(y)|^{\frac 12}}{|x-y|^{d-2}} u(y)dy\biggr|^2dx\\
        &\leq \int_{\R^d}|V(x)| \biggl( C_d\int_{\R^d}\frac{|V(y)|}{|x-y|^{d-2}} dy\biggr) \biggl( C_d\int_{\R^d}\frac{|u(y)|^2}{|x-y|^{d-2}} dy\biggr) dx\\
        &\leq C_d \Vert V \Vert_{\KKK} \int_{\R^d}\frac{C_d|V(x)|}{|x-y|^{d-2}}|u(y)|^2dydx\\
        &\leq \bigl(C_d \Vert V \Vert_{\KKK}\bigr)^2\Vert u \Vert_{L_x^2(\R^d)}^2.
    \end{align*}
    From the $TT^{*}$ argument \cite{gv}, we set $T=|V|^{\frac 12}|\nabla|^{-1}$.
    Then we find 
    \begin{align}\label{vvvu}
        \int_{\R^d} V |u|^2 dx=\bigl\Vert |V|^{\frac 12}|\nabla|^{-1} \nabla u \bigr\Vert_{L_x^2(\R^d)}^2
        \leq C_d \Vert V \Vert_{\KKK} \Vert \nabla u \Vert_{L_x^2(\R^d)}^2.
    \end{align}
    In particular, if $\Vert V_{-} \Vert_{\KKK} <\frac{1}{C_d}$, combining with \eqref{vvvu}, we have
    \begin{align*}
        (1- C_d\Vert V_{-} \Vert_{\KKK})\Vert \nabla u \Vert_{L_x^2(\R^d)}^2 
        \leq \Vert \nabla u \Vert_{L_x^2(\R^d)}^2 + \Vert V^{\frac 12} u \Vert_{L_x^2(\R^d)}^2 
        \leq (1 + C_d\Vert V \Vert_{\KKK})\Vert \nabla u \Vert_{L_x^2(\R^d)}^2. 
    \end{align*}
    Then the proof is completed.
\end{proof}

\subsection{Variational analysis}
In view of the proof in \cite{hong}, we can compute the sharp constant $C_{GN}$ for Gagliardo-Nirenberg inequality, which is crucially used to our later proof.
It can be estimated as follows.
\begin{equation}\label{p<c}
    P(u) \leq C_{GN}\Vert u \Vert_{\dot{H}_V^1(\R^d)}^{\gamma} \Vert u \Vert_{L_x^2(\R^d)}^{4-\gamma},
\end{equation}
where 
\begin{align*}
    C_{GN}=\sup_{u\in H_x^1(\R^d) \backslash \{0\}}W(u)
    =\sup_{u\in H_x^1(\R^d) \backslash \{0\}} \frac{P(u)}{\Vert u \Vert_{\dot{H}_V^1(\R^d)}^{\gamma} \Vert u \Vert_{L_x^2(\R^d)}^{4-\gamma}}.
\end{align*}

We know that $\QQQ$ is a strong solution in \cite{hong}.
Rewrite \eqref{p<c} as 
\begin{align*}
    \bigl[P(u)\bigr]^{\frac{2}{\gamma}}
    \leq C_{\QQQ}\Vert u \Vert_{\dot{H}_V^1(\R^d)}^2 \bigl(\Vert u \Vert_{L_x^2(\R^d)}^2\bigr)^{\frac{4-\gamma}{\gamma}},
\end{align*}
where 
\begin{align*}
    C_{\QQQ}=(C_{GN})^{\frac{2}{\gamma}}=\frac{\bigl[P(\QQQ) \bigr]^{\frac{2}{\gamma}}}{\Vert \QQQ \Vert_{\dot{H}_V^1(\R^d)}^{2} \Vert \QQQ \Vert_{L_x^2(\R^d)}^{\frac{2(4-\gamma)}{\gamma}}}=\frac{4^{\frac{2}{\gamma}}}{(4-\gamma)^{\frac{2-\gamma}{\gamma}}\gamma \Vert \QQQ \Vert_{L_x^2(\R^d)}^{\frac{4}{\gamma}}} .
\end{align*}

Since our paper focus on the long-time dynamical behavior of \eqref{NLHv} beyond the mass-energy threshold, the ground state $\QQQ$ is vital to the proof.
We show the properties of $\QQQ$ as follows.
\begin{proposition}[Pohozhaev identities]\label{Prelation}
    For  $\QQQ$ in the \eqref{elliptic1}, we have 
    \begin{align*}
        \Vert \QQQ \Vert_{\dot{H}_V^1(\R^d)}^2=\frac{\gamma w_{\QQQ}^2}{4-\gamma}\Vert \QQQ \Vert_{L_x^2(\R^d)}^2, 
        \quad P(\QQQ)=\frac{4w_{\QQQ}^2}{4-\gamma}\Vert \QQQ \Vert_{L_x^2(\R^d)}^2,
    \end{align*}
    and then 
    \begin{align}\label{pe}
         P(\QQQ)=\frac{8}{\gamma-2}E_V(\QQQ).
    \end{align}   
\end{proposition}
\begin{proof}
    Let $Q_w$ be a strong solution to the equation
    \begin{equation}\label{Q1}
        (\Delta-V)Q_w - w^2 Q_w + (|\cdot|^{-\gamma} \ast |Q_w|^2)Q_w=0,
    \end{equation}
    (\romannumeral1) Multiplying \eqref{Q1} by $Q_w$ and integrating by parts, we have 
    $$\int_{\R^d}(\Delta-V)  Q_w  \cdot Q_w dx- w^2 \int_{\R^d}Q_w^2dx + \int_{\R^d}(|\cdot|^{-\gamma} \ast |Q_w|^2)Q_w^2dx=0.$$
    We find
    \begin{align*}
        \int_{\R^d}(\Delta-V)  Q_w  \cdot Q_w dx
        =-\int_{\R^d}|\nabla Q_w|^2dx- \int_{\R^d}VQ_w^2dx
        =-||Q_w||_{\dot{H}_V^1(\R^d)}^2,
    \end{align*}
    Then we obtain
    \begin{equation}\label{multiply1}
        P(Q_w)
        = ||Q_w||_{\dot{H}_V^1(\R^d)}^2 + w^2 ||Q_w||_{L_x^2(\R^d)}^2
    \end{equation}
    \\
    (\romannumeral2) Multiplying \eqref{Q1} by $x \cdot \nabla Q_w $ and integrating by parts, we have 
    $$\int_{\R^d}(\Delta-V) Q_w \cdot x \cdot \nabla Q_w dx-\int_{\R^d}Q_w \cdot x \cdot \nabla Q_w dx+\int_{\R^d}(|\cdot|^{-\gamma} \ast |Q_w|^2)Q_w\cdot x \cdot \nabla Q_w dx =0.$$
    For each part, we have
    \begin{align*}
    &(a)~\int_{\R^d}(\Delta-V) Q_w \cdot x \cdot \nabla Q_w dx
    =\frac {d-2}{2} ||\nabla Q_w||_{L_x^2(\R^d)}^2 +\frac12 \int_{\R^d} ( x \cdot \nabla V +2 V)|Q_w|^2dx,\\
    &(b)~\int_{\R^d}w^2 Q_w \cdot x \cdot \nabla Q_w dx = -\frac {dw^2}{2}||Q_w||_{L_x^2(\R^d)}^2,\\
    &(c)~\int_{\R^d}(|\cdot|^{-\gamma} \ast |Q_w|^2)Q_w\cdot x \cdot \nabla Q_w dx
    =(\frac{\gamma}{4}-\frac d2)P(Q_w).
    \end{align*}
    Collecting them all, we have 
    \begin{align}\label{Q2}
        (d-\frac{\gamma}{2})P(Q_w)=
        (d-2)||Q_w||_{\dot{H}_V^1(\R^d)}^2+dw^2||Q_w||_{L_x^2(\R^d)}^2+\int_{\R^d}(2V + x \cdot \nabla V)|Q_w|^2dx.
    \end{align}
    Combining \eqref{Q1} and \eqref{Q2} together, we find 
    \begin{equation}\label{pohozhaev0}
        \left\{
        \begin{aligned}
        &||Q_w||_{\dot{H}_V^1(\R^d)}^2 = \frac{\gamma}{4-\gamma}w^2||Q_w||_{L_x^2(\R^d)}^2 + \frac{2}{4-\gamma} \int_{\R^d}(2V + x \cdot \nabla V)|Q_w|^2dx,\\
        &P(Q_w) = \frac{4}{4-\gamma}w^2||Q_w||_{L_x^2(\R^d)}^2 + \frac{2}{4-\gamma} \int_{\R^d}(2V + x \cdot \nabla V)|Q_w|^2dx.
        \end{aligned}
        \right.
    \end{equation}
    Indeed, if $V=0$, then 
    \begin{equation}
        \left\{
        \begin{aligned}
        &||\nabla Q_w||_{L_x^2(\R^d)}^2 = \frac{\gamma}{4-\gamma}w^2||Q_w||_{L_x^2(\R^d)}^2,\\
        &P(Q_w) = \frac{4}{4-\gamma}w^2||Q_w||_{L_x^2(\R^d)}^2.
        \end{aligned}
        \right.
    \end{equation}
    Let $\QQQ$ be the ground state given in Proposition \ref{Pq}, 
    and $$w_{\QQQ}^2=\frac{(4-\gamma)~||\QQQ||_{\dot{H}_V^1(\R^d)}^2}{\gamma~||\QQQ||_{L_x^2(\R^d)}^2}.$$ 
    From \eqref{pohozhaev0} we find
    \[
    \int_{\R^d}(2V + x \cdot \nabla V)|\QQQ|^2dx = 0.
    \]
    Thus we have
    \begin{align*}
        \Vert \QQQ \Vert_{\dot{H}_V^1(\R^d)}^2=\frac{\gamma w_{\QQQ}^2 }{4-\gamma}\Vert \QQQ \Vert_{L_x^2(\R^d)}^2, 
        \quad P(\QQQ)=\frac{4 w_{\QQQ}^2}{4-\gamma}\Vert \QQQ \Vert_{L_x^2(\R^d)}^2.
    \end{align*}
\end{proof}

\begin{remark}
    If $V_{-}=0$, according to Proposition \ref{Pq}, the ground state $Q$ satisfies \eqref{elliptic}. By the similar method, we find the pohozhaev identity 
    \begin{align*}
        \Vert Q \Vert_{\dot{H}_V^1(\R^d)}^2=\frac{\gamma}{4-\gamma}\Vert Q \Vert_{L_x^2(\R^d)}^2,  \quad P(Q)=\frac{4}{4-\gamma}\Vert Q \Vert_{L_x^2(\R^d)}^2.
    \end{align*}
    and then 
    \begin{align*}
         P(Q)=\frac{8}{\gamma-2}E_V(Q).
    \end{align*}
\end{remark}


\subsection{Virial identity} 
We will discuss the  virial identity for that our main results are closely based on the property of it. 
We have
    \begin{align*}
        I(t)=\int_{\R^d}|x|^2|u(t,x)|^2dx,
    \end{align*}
for $xu \in L_x^2(\R^d)$.

A natural question is that we want to make some estimates for $I(t)$. 
By the accurate calculation, we summarize  the first and second derivatives of $I(t)$ as follows, which will be used throughout our proof.
\begin{lemma}\label{Lviral}
    Assume that $u(t,x)$ is the solution to \eqref{NLHv}. 
    Then we have
    \begin{align*}
        I'(t)=4\Im \int_{\R^d}x\bar{u}\nabla udx,
    \end{align*}
    and
    \begin{align*}
        I''(t)=&~8\Vert \nabla u \Vert_{L_x^2(\R^d)}^2-4\int_{\R^d}x\cdot \nabla V |u|^2dx-2\gamma P(u)\\
        =&~8\Vert u\Vert_{\dot{H}_V^1(\R^d)}^2-4\int_{\R^d}x\cdot \nabla V |u|^2dx-8\int_{\R^d}V|u|^2dx-2\gamma P(u)\\
        =&~8\Vert u\Vert_{\dot{H}_V^1(\R^d)}^2-2\gamma P(u)-e(t),
    \end{align*}
    where $$e(t) := 4\int_{\R^d}(2V+x\cdot \nabla V)|u|^2dx.$$
\end{lemma}

Indeed, we incorporate all terms involving the potential $V$ into $e(t)$. 
And $e(t)$ is  entirely dependent on $2V+x\cdot \nabla V$, which have been provided  in Theorem \ref{Tconclusion}.
So we can primarily focus on estimating the rest terms.   
We denote that 
\begin{align}\label{i+e}
    \widetilde{I''(t)}=I''(t)+e(t)=8\Vert u\Vert_{\dot{H}_V^1(\R^d)}^2-2\gamma P(u).
\end{align}
Rewrite \eqref{energy} as 
$$E_V(u)=\frac{1}{2}\Vert u \Vert_{\dot{H}_V^1(\R^d)}^2-\frac{1}{4}P(u).$$ 
Then combining the above two equations, we find 
\begin{align}\label{pE_V}
    P(u) = \frac{1}{2(\gamma-2)}\bigl(16E_V(u)-\widetilde{I''(t)}\bigr)
\end{align}
and
\begin{align}\label{uE_V}
    \Vert u \Vert_{\dot{H}_V^1(\R^d)}^2 = \frac{8\gamma E_V(u)-\widetilde{I''(t)}}{4(\gamma -2)}.
\end{align}

Next, according to \eqref{pE_V} and \eqref{uE_V}, we come to find the relation between $I'(t)$, $P(u)$ and $\Vert u \Vert_{\dot{H}_V^1(\R^d)}^2$, which is important to our proof of Theorem \ref{Tconclusion}.
\begin{lemma}\label{Lim2}
    \begin{align*}
        \biggl(\Im \int x \bar{u} \nabla u dx \biggr)^2 \leq \int |x|^2 |u|^2dx \biggl(\Vert u \Vert_{\dot{H}_V^1(\R^d)}^2-\frac{[P(u)]^{\frac{2}{\gamma}}}{C_\QQQ[M(u)]^{\frac{4-\gamma}{\gamma}}} \biggr).
    \end{align*}
\end{lemma}
\begin{proof}
    For $e^{i \lambda |x|^2}u$, we can compute that 
    \begin{align*}
        \bigl|\nabla (e^{i \lambda |x|^2}u)\bigr|^2 = & ~\bigl[ (\nabla e^{i \lambda |x|^2})u + e^{i \lambda |x|^2}\nabla u\bigr]^2\\
        = &~ \bigl[ (\nabla e^{i \lambda |x|^2})u\bigr]^2 + 2 (\nabla e^{i \lambda |x|^2})u \cdot e^{i \lambda |x|^2}\nabla u + \bigl[ e^{i \lambda |x|^2}\nabla u\bigr]^2\\
        = &~ \bigl[ e^{i \lambda |x|^2} \cdot \nabla (i \lambda |x|^2) \cdot u \bigr]^2 + 2e^{i \lambda |x|^2} \cdot \nabla (i \lambda |x|^2) \cdot u \nabla u + (\nabla u )^2\\
        =& ~\bigl[ i\lambda (\nabla |x|^2) u \bigr]^2+2i\lambda \nabla |x|^2 \cdot u \cdot \nabla u + (\nabla u )^2\\
        =&~\bigl[ 2\lambda(\Im \nabla x \cdot \bar{x}) \cdot u\bigr]^2+4\lambda (\Im \nabla x \cdot \bar{x}) \cdot u \cdot \nabla u + (\nabla u )^2.
    \end{align*}
    Then we have the $\dot{H}_V^1$ norm of $e^{i \lambda |x|^2}u$, that is 
    \begin{align*}
        \Vert e^{i \lambda |x|^2}u \Vert_{\dot{H}_V^1(\R^d)} 
        = &~ \int \bigl|\nabla (e^{i \lambda |x|^2}u)\bigr|^2 dx + \int V(x) \bigl|e^{i \lambda |x|^2}u\bigr|^2 dx\\
        =& ~4\lambda^2\int |x|^2 |u|^2dx+4\lambda \Im \int x\bar{u} \nabla u dx + \int (|\nabla u|^2 + V|u|^2)dx\\
        =& ~4\lambda^2\int |x|^2 |u|^2dx+4\lambda \Im \int x\bar{u} \nabla u dx + \Vert u \Vert_{\dot{H}_V^1(\R^d)}^2.
    \end{align*}
    We substitute $e^{i \lambda |x|^2}u$ for $u$ in \eqref{p<c},  that is
    \begin{align*}
        P(e^{i \lambda |x|^2}u) \leq C_{GN} \bigl\Vert e^{i \lambda |x|^2}u \bigr\Vert_{\dot{H}_V^1(\R^d)}^{\gamma} \bigl\Vert e^{i \lambda |x|^2}u \bigr\Vert_{L_x^2(\R^d)}^{4-\gamma}.
    \end{align*}
    Thus, according to $  C_\QQQ=(C_{GN})^{\frac{2}{\gamma}}$, 
    \begin{align}\label{cq-p>0}
      C_\QQQ \Vert u \Vert_{L_x^2(\R^d)}^{\frac{2(4-\gamma)}{\gamma}}\bigl[4\lambda^2\int |x|^2 |u|^2dx+4\lambda \Im \int x\bar{u} \nabla u dx + \Vert u \Vert_{\dot{H}_V^1(\R^d)}^2 \bigr]-  [P(u)]^{\frac{2}{\gamma}} \geq 0
    \end{align}
  for any $\lambda \in \R$.  We find that the left side of \eqref{cq-p>0} is a quadratic function in $\lambda$.  The discriminant of this function in $\lambda$ must be negative, which yields Lemma  \ref{Lim2}.  
\end{proof}
\begin{remark}
    In order to describe more visually, 
we rewrite Lemma \ref{Lim2} by \eqref{pE_V} and \eqref{uE_V}. 
It follows that 
\begin{align}\label{i'2}
    [I'(t)]^2 \leq 16 I(t)\biggl[ \frac{8\gamma E_V(u)-\widetilde{I''(t)}}{4(\gamma-2)} - \frac{1}{C_\QQQ[M(u)]^{\frac{4-\gamma}{\gamma}}}\biggl(  \frac{16E_V(u)-\widetilde{I''(t)}}{2(\gamma-2)} \biggr)^{\frac{2}{\gamma}} \biggr].
\end{align}
We denote that 
\begin{align*}
    f(x)= \frac{2\gamma }{\gamma-2}E_V(u) - \frac{1}{4(\gamma-2)}x- \frac{1}{C_\QQQ[M(u)]^{\frac{4-\gamma}{\gamma}}}\biggl(  \frac{16E_V(u)-x}{2(\gamma-2)} \biggr)^{\frac{2}{\gamma}}
\end{align*}
for any $x\in \left( -\infty, 16E_V(u) \right]$.
Thus, \eqref{i'2} can be simplified as 
\begin{align}\label{i'2new}
    [I'(t)]^2 \leq 16 I(t)f(\widetilde{I''(t)}).
\end{align}
\end{remark}

\section{Proof of the Main Theorem}
In this section, we will prove the main results in Theorem \ref{Tconclusion}. 

Let $z(t)=\sqrt{I(t)}$. By \eqref{i'2new}, we have
\begin{align}\label{z'2<4i''}
    [z'(t)]^2=\frac{[I'(t)]^2}{4I(t)} \leq 4f(\widetilde{I''(t)}).
\end{align}
It indicates that $z'(0)$ can be estimated by $f(\widetilde{I''(t)})$. 
Actually,  the condition $I'(0)$ in Theorem \ref{Tconclusion} is closely related to $z'(0)$. 
Thus, by \eqref{z'2<4i''}, we first need to discuss the property of $f(x)$ to pave the way for proving our main results.

We can compute that 
\begin{align*}
    f'(x) =- \frac{1}{4(\gamma-2)}+\frac{1}{C_\QQQ[M(u)]^{\frac{4-\gamma}{\gamma}}} \cdot \frac{2}{\gamma} \cdot \bigl[\frac{1}{2(\gamma-2)} \bigr]^{\frac{2}{\gamma}}\bigl[16E_V(u)-x\bigr]^{\frac{2}{\gamma}-1}.
\end{align*}
Since $\frac{2}{\gamma}-1 < 0 ~(s_c>0)$, $f(x)$ is decreasing on $(-\infty,x_0)$ and increasing on $(x_0,16E_V(u))$, where $x_0$ satisfies
\begin{align}\label{x_0}
    \frac{1}{4(\gamma-2)}=\frac{1}{C_\QQQ[M(u)]^{\frac{4-\gamma}{\gamma}}} \cdot \frac{2}{\gamma} \cdot \bigl[\frac{1}{2(\gamma-2)} \bigr]^{\frac{2}{\gamma}}\bigl[16E_V(u)-x_0\bigr]^{\frac{2}{\gamma}-1}.
\end{align}
Then 
\begin{align*}
    f(x_0) =&~ \frac{2\gamma }{\gamma-2}E_V(u) - \frac{1}{4(\gamma-2)}x_0- \frac{1}{C_\QQQ[M(u)]^{\frac{4-\gamma}{\gamma}}}\biggl(  \frac{16E_V(u)-x_0}{2(\gamma-2)} \biggr)^{\frac{2}{\gamma}}\\
    =&~ \frac{2\gamma }{\gamma-2}E_V(u) - \frac{1}{4(\gamma-2)}x_0 - \frac{1}{4(\gamma-2)}\cdot \frac{\gamma}{2} \cdot (16E_V(u)-x_0)\\
    =&~  \frac{x_0}{8}.
\end{align*}

Using \eqref{pE_V}, \eqref{uE_V} and the expression of $C_\QQQ$, we rewrite \eqref{x_0} as 
\begin{align}\label{meps1}
    \biggl( \frac{M(u)}{M(\QQQ)} \biggr)^{1-s_c} \biggl( \frac{E_V(u)-\frac{1}{16}x_0}{E_V(\QQQ)} \biggr)^{s_c} 
    = \MMM\EEE \biggl( 1- \frac{x_0}{16E_V(u)} \biggr)=1
\end{align}
for  $x_0\in \left( -\infty, 16E_V(u) \right]$.
Then  $\MMM\EEE >1$ is equivalent to $x_0 \geq 0$.
We can also rewrite \eqref{me<1} as 
\begin{align}\label{z'2}
    [z'(0)]^2 \geq \frac{x_0}{2}=4f(x_0).
\end{align}

\subsection{The proof of blow up.} 
Using the Hardy inequality 
\[\int_{0}^\infty\biggl[\frac{1}{x}\int_{0}^xg(\xi)d\xi\biggr]^pdx\leq \biggl(\frac{p}{p-1}\biggr)^p\int_{0}^\infty g^p(x)dx, \quad g(x)\geq 0, \quad p>1, \]
and the conservation of mass, we have
\begin{align*}
    \Vert u_0 \Vert_{L_x^2(\R^d)}^2=\int_{\R^d}|u(x,t)|^2dx
    \leq \Vert xu(t) \Vert_{L_x^2(\R^d)}\Vert u(t)\Vert_{\dot{H}_x^1(\R^d)} \to 0, 
    \quad t \to T_+(u),
\end{align*}
where  we have used the fact that  $u_0 \in \Sigma$  implies  the corresponding solution $u$ belongs to $\Sigma$.
On the basis of  Glassey's classical argument \cite{glassey}, our proof of blow up is to prove
\begin{align*}
    I''(t) < 0, \quad \forall ~ t \in \left[ 0,T_+(u) \right).
\end{align*}
To that end, we  demonstrate $z''(t)$ for that  $z(t)=\sqrt{I(t)}$.  

Firstly, we  make some  equivalent transformation to the conditions about the blow up in  Theorem \ref{Tconclusion}. 
The assumption $I'(0)\leq 0$ implies $z'(0)\leq 0$.
According to  \eqref{pe} and  \eqref{meps1},  the  assumption \eqref{mpu>mpq} implies that
\begin{align*}
   \biggr[ \frac{M(u_0)}{M(\QQQ)} \biggl]^{1-s_c} \biggr[ \frac{\frac{\gamma-2}{8}P(u_0)}{E_V(\QQQ)} \biggl] 
   > 1 = \biggr[ \frac{M(u)}{M(\QQQ)} \biggl]^{1-s_c} \biggl[ \frac{E_V(u)-\frac{1}{16}x_0}{E_V(\QQQ)} \biggr]^{s_c}. 
\end{align*}
Then by \eqref{pE_V},  we get 
\begin{align}\label{i''0<x0}
    \widetilde{I''(0)}< x_0.
\end{align}
Since $2V+x\cdot \nabla V \geq 0$, we find $e(0) \geq 0$. Then $\widetilde{I''(0)}=I''(0)+e(0) < x_0$ yields that $I''(0)<x_0$. Thus, we have
\begin{align*}
    z''(0)=&~\frac{1}{z(0)}\bigl( \frac{I''(0)}{2} -(z'(0))^2 \bigr)\\
    <&~\frac{1}{z(0)}\bigl( \frac{x_0}{2} -(z'(0))^2 \bigr)\\
    \leq&~ \frac{1}{z(0)}\bigl(  \frac{x_0}{2}- \frac{x_0}{2} \bigr)=0.
\end{align*}

 We ~claim ~that
 \begin{align}\label{Cz''0}
\bm{z''(t)<0,\quad \forall ~ t \in \left[ 0,T_+(u) \right)} . 
\end{align}

Indeed, if the claim holds, we assume that   $T_+(u)= \infty$.
By $z'(0)<0$ and $z''(t)<0$,  we have $z'(1)<0$. Then
\begin{align*}
    z(t)=&~\int_1^t z'(s)ds+z(1)\\
    \leq &~\int_1^t z'(1)ds+z(1)\\
    \leq &~z'(1)(t-1)+z(1).
\end{align*}
For $t \to \infty$, $z(t)<0$ is contradicted to $z(t)>0$. Thus $T_+(u)<\infty$, which implies that $z(t)$ will approach $0$ in a finite time. 

Now we come to prove the claim \eqref{Cz''0}.
If the claim does not hold,  there exists $t_0=\sup \{t\in (0, T_+(u)),z''(0)\geq 0\}$ satisfying $z''(t)<0$ for any $t\in\left[ 0,t_0 \right)$. And by the continuity of $z''(t)$, we have $$z''(t_0)=0.$$
Using $z'(0) \leq 0$ and \eqref{z'2}, we have
\begin{align*}
    z'(t)<z'(0)\leq 0, \quad \forall ~ t \in [0,t_0],
\end{align*}
and
\begin{align*}
    4f(x_0) \leq [z'(0)]^2 \leq [z'(t)]^2, \quad \forall ~ t \in [0,t_0].
\end{align*}
Combining with \eqref{z'2<4i''}, we obtain $f(\widetilde{I''(t)}) > f(x_0)$ for any $t \in [0,t_0]$. 
Then $\widetilde{I''(t)} \neq x_0$ for any $t \in [0,t_0]$. 
According to \eqref{i''0<x0} and the continuity of $\widetilde{I''(t)}$, we have 
\begin{align}\label{i''t<x0}
    \widetilde{I''(t)} < x_0, \quad \forall ~ t \in [0,t_0].
\end{align}
Furthermore, $2V+x\cdot \nabla V \geq 0$ yields that $e(t) \geq 0$. From \eqref{i+e}, We find  
\begin{align*}
     I''(t) = \widetilde{I''(t)} - e(t)<x_0, \quad \forall ~ t \in [0,t_0].
\end{align*}
Therefore, 
\begin{align*}
    z''(t_0)=&~\frac{1}{z(t_0)}\bigl( \frac{I''(t_0)}{2} -(z'(t_0))^2 \bigr)\\
    <&~\frac{1}{z(t_0)}\bigl( \frac{x_0}{2} -(z'(t_0))^2 \bigr)\\
    \leq&~ \frac{1}{z(t_0)}\bigl(  \frac{x_0}{2}- \frac{x_0}{2} \bigr)=0,
\end{align*}
which is contradicted to $z''(t_0)=0$.
Thus the claim holds.

Consequently, the proof of Part $(i)$ in Theorem \ref{Tconclusion}
 is completed.
 
\subsection{Global existence.}
We first convert the conditions related to the global existence equivalently. The assumption $I'(0)\geq 0$  implies  $z'(0)\geq 0$.  The assumption \eqref{mpu<mpq} implies 
$$\widetilde{I''(0)}>x_0.$$
Furthermore,  $2V+x\cdot \nabla V \leq 0$ ensures that $e(t)\leq 0$. Then $$I''(x_0) = \widetilde{I''(0)}-e(0) >x_0.$$

Using \eqref{z'2} and $z'(0) \geq 0$, we derive
\begin{align}\label{sqrtx_0/2}
    z'(0)\geq \sqrt{\frac{x_0}{2}}.
\end{align}
This leads to the existence of $t_1 \geq 0$ such that 
\begin{align}\label{z't1}
    z'(t_1) > \sqrt{\frac{x_0}{2}}=2\sqrt{f(x_0)}.
\end{align}
If $ z'(0)$ strictly exceeds $\sqrt{\frac{x_0}{2}}$, we can choose $t_1=0$.  If $z'(0)=\sqrt{\frac{x_0}{2}}$, we find
\begin{align*}
    z''(0)=~\frac{1}{z(0)}\bigl( \frac{I''(0)}{2} -(z'(0))^2 \bigr)
    >~\frac{1}{z(0)}\bigl( \frac{x_0}{2} -(z'(0))^2 \bigr)
    =0.
\end{align*}
Then for small $t_1>0$, we can select a small parameter $\eps_1>0$ such that
\begin{align*}
    z'(t_1)=2\sqrt{f(x_0)}+2\eps_1.
\end{align*}

\textbf{We claim that}
\begin{align}\label{Cz'}
    \bm{z'(t) > 2\sqrt{f(x_0)}+2\eps_1,\quad \forall ~ t\geq t_1}.
\end{align}

 Indeed, if the claim does not hold, there exists $t_2= \inf \{t>t_1:z'(t)\leq 2\sqrt{f(x_0)}+\eps_1\}$.
By the continuity of $z'(t)$, we have
\begin{align}\label{z't2}
    z'(t_2)=2\sqrt{f(x_0)}+\eps_1
\end{align}
and
\begin{align}\label{z'int1t2}
    z'(t) \geq 2\sqrt{f(x_0)}+\eps_1, \quad \forall ~ t \in [t_1,t_2].
\end{align}
According to \eqref{z'2<4i''}, we rewrite \eqref{z'int1t2} as
\begin{align}\label{2fx_0}
    (2\sqrt{f(x_0)}+\eps_1)^2 \leq (z'(t))^2 \leq 4f(\widetilde{I''(t)}), \quad \forall ~ t \in [t_1,t_2],
\end{align}
we find that $f(\widetilde{I''(t)})>f(x_0)$ for all $t \in [t_1,t_2]$.
Hence
\begin{align*}
    \widetilde{I''(t)}\neq x_0, \quad \forall ~ t \in [t_1,t_2].
\end{align*}
Combining with  $\widetilde{I''(0)}>x_0$, 
we get 
\begin{align*}
   \widetilde{I''(t)}>x_0, \quad \forall ~ t \in [t_1,t_2]. 
\end{align*}

Then again, there could exist a constant $C$ such that
\begin{align}\label{i''c}
    \widetilde{I''(t)} \geq x_0 + \frac{\sqrt{\eps_1}}{C},  \quad \forall ~ t \in [t_1,t_2].
\end{align}
If $\widetilde{I''(t)} \geq x_0 +1$, then \eqref{i''c} holds for $C$ large enough. 
If $x_0< \widetilde{I''(t)} \leq x_0+1$, by the Taylor equation of $f$ around $x=x_0$, there exists $a>0$ such that
\begin{align*}
    f(x)=f(x_0)+a(x-x_0)^2~\rm{when}~|\mathit{x-x}_0|\leq 1.
\end{align*}
Substituting this equality for \eqref{2fx_0} with $x=\widetilde{I''(t)}$, we have 
\begin{align}\label{i''cnew}
    (2\sqrt{f(x_0)}+\eps_1)^2 \leq (z'(t))^2 \leq 4f(x_0)+4a(\widetilde{I''(t)}-x_0)^2.
\end{align}
Combining  \eqref{i''c} and \eqref{i''cnew}, we obtain $C=2\sqrt{a}(4(f(x_0))^{\frac{1}{2}}+\eps_1)^{-\frac{1}{2}}$.
Thus \eqref{i''c} holds.

By \eqref{i''c} and $e(t) \leq 0$, we find 
\begin{align}\label{i''c1}
    x_0 + \frac{\sqrt{\eps_1}}{C} \leq \widetilde{I''(t)} = I''(t) +e(t) \leq I''(t), \quad \forall ~ t \in [t_1,t_2].
\end{align}
However, by \eqref{z't2} and \eqref{i''c1} we have
\begin{align*}
    z''(t_2)=&~\frac{1}{z(t_2)}\bigl(  \frac{I''(t_2)}{2}-(z'(t_2))^2  \bigr)\\
    \geq &~ \frac{1}{z(t_2)}\bigl(  \frac{\sqrt{\eps_1}}{2C} -4\eps_1 \sqrt{f(x_0)} -\eps_1^2 \bigr)\\
    >&~\frac{1}{z(t_2)}\frac{\sqrt{\eps}}{4C},
\end{align*}
where $\eps<\eps_1$ is small enough. Then we get $z''(t_2)>0$, which contradicts with \eqref{z't2} and \eqref{z'int1t2}.
So we obtain the claim.

Next we use the claim \eqref{Cz'} to prove \eqref{sup} in Theorem \ref{Tconclusion}.
We note that \eqref{i''c} holds for all $t \in \left[ t_1,T_+(u) \right)$. Hence, we obtain
\begin{align*}
    [M(u)]^{1-s_c}[P(u)]^{s_c}=&~[M(u)]^{1-s_c}[\frac{1}{2(\gamma-2)}(16E_V(u))-\widetilde{I''(t)}]^{s_c}\\
    \leq &~ [M(u)]^{1-s_c} [\frac{1}{2(\gamma-2)}(16E_V(u))-x_0 - \frac{\sqrt{\eps_1}}{C}]^{s_c}\\
    <&~ [M(u)]^{1-s_c} [\frac{1}{2(\gamma-2)}(16E_V(u)-x_0 )]^{s_c}\\
    =&~[M(\QQQ)]^{1-s_c}[P(\QQQ)]^{s_c}.
\end{align*}
Then by mass and energy conservation, we have
\begin{align*}
    \Vert u \Vert_{\dot{H}_V^1(\R^d)} = 2E_V(u)+\frac{1}{2}P(u) 
    \leq 2E_V(u)+\frac{M(\QQQ)^{\frac{1-s_c}{s_c}}P(\QQQ)}{M(u_0)^{\frac{1-s_c}{s_c}}}
    <A
\end{align*}
for all $t\in\left[ t_1,T_+(u) \right)$, where constant $A$ depending on $M(u_0)$, $E_V(u_0)$, $M(\QQQ)$ and $E_V(\QQQ)$.
So $u(t,x)$ exists globally.\\

\noindent{\bf Acknowledgement} The author Jing Lu was supported by the National Natural Science Foundation of China (No. 12101604).

\end{document}